\documentclass[12pt,a4paper]{article}

\usepackage[utf8]{inputenc}
\usepackage[T1]{fontenc}
\usepackage{lmodern}
\usepackage[english]{babel}
\usepackage{amsmath,amssymb,amsthm,mathtools}
\usepackage{geometry}
\usepackage{enumitem}
\usepackage{graphicx}
\usepackage{tikz}
\usepackage{float}
\usetikzlibrary{calc}
\usepackage[hidelinks]{hyperref}
\theoremstyle{thmstyleone}
\newtheorem{theorem}{Theorem}[section]
\newtheorem{proposition}[theorem]{Proposition}
\newtheorem{lemma}[theorem]{Lemma}
\newtheorem{corollary}[theorem]{Corollary}
\newtheorem{conjecture}[theorem]{Conjecture}

\theoremstyle{thmstylethree}

\theoremstyle{thmstyletwo}
\newtheorem{remark}[theorem]{Remark}

\newcommand{\HH}{\mathbb H}
\newcommand{\SSph}{\mathbb S}
\newcommand{\RR}{\mathbb R}
\newcommand{\Vol}{\operatorname{Vol}}

\newcommand{\st}{\operatorname{st}}
\newcommand{\lk}{\operatorname{lk}}

\AtBeginDocument{\raggedbottom}

\begin{document}	

\title{Minimal-Volume Equiangular Hyperbolic $4$-Polytopes}

\author{Andrey Egorov\\
	Sobolev Institute of Mathematics, Novosibirsk, Russia\\
	\texttt{a.egorov2@g.nsu.ru}}
\date{}
\maketitle

\begin{abstract}
We study finite-volume convex hyperbolic $4$-polytopes whose
dihedral angles are all equal to a fixed strictly acute angle $\alpha$.
Put
$\alpha_0=\arccos(1/3)$ and $\alpha_1=\arccos(1/4)$.  We first show that the
class is empty for $\alpha<\alpha_0$.  For
$\alpha_0\leq\alpha<\alpha_1$, the unique polytope of minimum volume is the
regular hyperbolic $4$-simplex with dihedral angle $\alpha$.  As
$\alpha$ increases to $\alpha_1$, this simplex degenerates to a Euclidean
one and no hyperbolic simplex exists at $\alpha_1$.  For
$\alpha_1\leq\alpha<\pi/2$, the unique minimum is the regular equiangular
hyperbolic $4$-cube.  The proof combines the hyperbolic Gram--Euler relation
with the Davis--Okun theorem on the Charney--Davis inequality for flag
triangulations of the $3$-sphere.  The geometric step is a missing-face
argument showing that the boundary of the dual of every nonsimplex in the
class is flag.
\end{abstract}



\maketitle

\section{Introduction}

The problem of minimizing volume under a prescribed angle condition is one
of the basic extremal questions for hyperbolic polyhedra and polytopes.  We
use \emph{equiangular} in its literal sense: all dihedral angles have one
common value.  The class studied in this paper is further restricted by the
inequality $\alpha<\pi/2$, and will be called the \emph{acute
equiangular class}.  Thus neither the right-angled case $\alpha=\pi/2$ nor
any obtuse angle is included. 

In dimension three, the non-obtuse equal-angle condition is governed by
Andreev's theorem, but volume minimization still requires substantial
geometric input. Three dimensional equiangular polyhedra may exist only for
$\frac{\pi}{3}\le \alpha\le \frac{\pi}{2}$.  At the left endpoint $\alpha=\pi/3$, Atkinson's estimate
implies that the regular ideal tetrahedron is the unique minimum among
$\pi/3$-equiangular hyperbolic polyhedra \cite{Atkinson2009}.  Further
arithmetic and volume information in the $\pi/3$-equiangular class was
obtained by Nonaka and Yoshida \cite{NonakaYoshida2026}.  At the excluded
right-angled endpoint $\alpha=\pi/2$, the minimum among finite-volume
right-angled hyperbolic polyhedra is the right-angled triangular bipyramid
$P(3,2)$ \cite{VesninEgorov2025}.

The intermediate tetrahedral range in dimension three was treated in
\cite{Egorov2026}.  Namely, for
\[
 \frac{\pi}{3}\leq\alpha<\arccos\frac13,
\]
the regular tetrahedron with dihedral angle $\alpha$ is the unique
minimum-volume equiangular hyperbolic polyhedron.  After the regular
tetrahedron disappears, the natural candidate is the equiangular
parallelepiped; this is posed as an open problem in \cite{Egorov2026}.

The acute four-dimensional problem has a different feature.  In even dimension,
the Gram--Euler relation expresses hyperbolic volume as an alternating sum of
face angles.  For an equiangular simple $4$-polytope, all local angles are
determined by the common dihedral angle.  The volume therefore becomes a
linear function of only the numbers of vertices and facets.  The required
sharp relation between these two numbers is supplied by the
Charney--Davis inequality for flag triangulations of $\SSph^3$, formulated
in \cite{CharneyDavis1993} and proved in this dimension by Davis and Okun
\cite{DavisOkun2001}.

There are two geometrically distinguished angles.  We put
\begin{equation}\label{eq:critical-angles}
 \alpha_0=\arccos\frac13=1.230959\ldots,
 \qquad
 \alpha_1=\arccos\frac14=1.318116\ldots.
\end{equation}
The first is the dihedral angle of a Euclidean regular tetrahedron.  It is
the ideal-vertex threshold for equiangular hyperbolic $4$-polytopes.  The
second is the dihedral angle of a Euclidean regular $4$-simplex, and is the
angle at which the hyperbolic regular $4$-simplex degenerates.  The regular
hyperbolic $4$-cube exists throughout
$[\alpha_0,\pi/2)$: it is ideal at $\alpha_0$, compact for
$\alpha>\alpha_0$, and shrinks to a point as $\alpha\to\pi/2$.

Our main result is the following.

\begin{theorem}[Main theorem]\label{thm:main}
Let $P\subset\HH^4$ be a finite-volume convex hyperbolic $4$-polytope whose
dihedral angles are all equal to an acute angle
$\alpha\in(0,\pi/2)$.
\begin{enumerate}
\item If
\[
 \alpha_0\leq\alpha<\alpha_1,
\]
then
\[
 \Vol(P)\geq\Vol(\Delta^4_\alpha),
\]
where $\Delta^4_\alpha$ is the regular hyperbolic $4$-simplex with dihedral
angle $\alpha$.  Equality holds only if $P$ is isometric to
$\Delta^4_\alpha$.
\item If
\[
 \alpha_1\leq\alpha<\frac{\pi}{2},
\]
then
\[
 \Vol(P)\geq\Vol(C^4_\alpha),
\]
where $C^4_\alpha$ is the regular equiangular hyperbolic $4$-cube.  Equality
holds only if $P$ is isometric to $C^4_\alpha$.
\end{enumerate}
\end{theorem}

The right-angled endpoint is deliberately excluded.  The family
$C^4_\alpha$ degenerates as $\alpha\to\pi/2$, whereas nondegenerate
right-angled hyperbolic $4$-polytopes form a different limiting problem.  For
example, among ideal right-angled hyperbolic $4$-polytopes, the ideal
$24$-cell is optimal both for volume and for the number of facets
\cite{Kolpakov2012}.

The proof has two independent parts.  First, the Gram matrix of a collection
of pairwise intersecting facets shows that the dual boundary of every
non-simplex is a flag simplicial $3$-sphere.  Second, the Gram--Euler relation
and the Davis--Okun inequality give a sharp comparison with the $4$-cube.
In the range where the simplex exists, a one-variable spherical
Schl\"afli calculation shows that the cube has strictly larger volume than
the simplex.

The paper is organized as follows.  Section~\ref{sec:geometry} describes the
angle range and the two model polytopes.  Section~\ref{sec:flag} proves the
flag dichotomy and records the required face-number inequalities.
Section~\ref{sec:volume} reduces volume to a linear expression and establishes
the analytic inequalities for its coefficients.  Section~\ref{sec:proof}
proves Theorem~\ref{thm:main}.  The final section formulates the natural
even-dimensional extension and explains why the present proof is special to
dimension four.

\section{Equiangular polytopes and the two models}\label{sec:geometry}

\subsection{Gram matrices}

We use the hyperboloid model in Minkowski space $\RR^{4,1}$ with scalar
product
\[
 \langle x,y\rangle=-x_0y_0+x_1y_1+\cdots+x_4y_4.
\]
Thus
\[
 \HH^4=\{x\in\RR^{4,1}:\langle x,x\rangle=-1,\ x_0>0\}.
\]
A supporting hyperplane of a convex polytope has the form
$u^\perp\cap\HH^4$, where $u$ is an outward unit spacelike normal.  If two
facets meet at interior dihedral angle $\theta$, then
\begin{equation}\label{eq:dihedral-Gram}
 \langle u_i,u_j\rangle=-\cos\theta.
\end{equation}
For disjoint supporting hyperplanes, the corresponding scalar product is at
most $-1$ after the outward normals are chosen.

We shall use two standard facts about polytopes with non-obtuse dihedral
angles.  First, let $G_I$ be the principal Gram submatrix belonging to a
collection of facets of such a hyperbolic polytope.  If $G_I$ is positive
definite, then these facets have a common face of codimension $|I|$.  If
$|I|=n$ and $G_I$ is indecomposable, positive semidefinite of rank $n-1$,
with a positive vector spanning its kernel, then the facets have a common
ideal vertex.  These are the elliptic and parabolic face criteria from
Andreev's intersection theorem \cite{Andreev1970Intersection}; see also
\cite[Part I, Chapter 6, Sections 2--3]{Vinberg1993}.  In particular, every
ordinary vertex of a non-obtuse hyperbolic $n$-polytope belongs to exactly
$n$ facets.  Second, every compact Euclidean polytope whose dihedral angles
are all strictly smaller than $\pi/2$ is a simplex \cite{Coxeter1934}.

\begin{proposition}\label{prop:simple-angle-range}
Let $P\subset\HH^4$ be a finite-volume acute equiangular polytope
with common dihedral angle $\alpha$.  Then $P$ is simple and
\[
 \alpha\geq\alpha_0=\arccos\frac13.
\]
If $\alpha=\alpha_0$, every vertex is ideal.  If
$\alpha>\alpha_0$, every vertex is finite and $P$ is compact.
\end{proposition}

\begin{proof}
An ordinary vertex belongs to four facets by the non-obtuse intersection
theorem.  The horospherical link of an ideal vertex is a compact Euclidean
$3$-polytope with all dihedral angles equal to $\alpha<\pi/2$.  By Coxeter's
Euclidean theorem, this link is a tetrahedron.  Hence an ideal vertex also
belongs to four facets, and $P$ is simple.

The link at a finite vertex is a regular spherical tetrahedron with dihedral
angle $\alpha$.  At an ideal vertex it is the corresponding Euclidean
tetrahedron.  Put $c=\cos\alpha$.  The Gram matrix of the four outward face
normals of either link is
\begin{equation}\label{eq:link-Gram}
 H_\alpha=(1+c)I_4-cJ_4.
\end{equation}
Its eigenvalues are
\[
 1+c\quad\text{with multiplicity }3,
 \qquad
 1-3c\quad\text{with multiplicity }1.
\]
A spherical link requires positive definiteness, and a Euclidean link
requires positive semidefiniteness of rank $3$.  Thus $c\leq1/3$, with
equality precisely at an ideal vertex.  This is equivalent to the asserted
angle range and the description of the vertices.
\end{proof}

We write $f_i(P)$ for the number of $i$-faces of $P$.  Since $P$ is simple,
\begin{equation}\label{eq:simple-f-relations}
 f_1(P)=2f_0(P),
 \qquad
 f_2(P)=f_0(P)+f_3(P).
\end{equation}
The first identity counts vertex-edge incidences and the second follows from
Euler's relation for the boundary $3$-sphere.

\subsection{The regular 4-simplex}

Let $\Delta^4_\alpha$ be a regular hyperbolic $4$-simplex with dihedral angle
$\alpha$.  The Gram matrix of its five outward facet normals is
\begin{equation}\label{eq:simplex-Gram}
 G^\Delta_\alpha=(1+c)I_5-cJ_5,
 \qquad c=\cos\alpha.
\end{equation}
Its eigenvalues are $1+c$ with multiplicity $4$ and $1-4c$ with
multiplicity $1$.  The hyperbolic signature criterion, together with the
vertex-link criterion in Proposition~\ref{prop:simple-angle-range}, gives the
following standard conclusion.

\begin{proposition}\label{prop:simplex-existence}
The finite-volume regular hyperbolic $4$-simplex $\Delta^4_\alpha$ exists
exactly for
\[
 \alpha_0\leq\alpha<\alpha_1.
\]
It is ideal at $\alpha=\alpha_0$, compact for
$\alpha_0<\alpha<\alpha_1$, and undergoes Euclidean degeneration as
$\alpha\to\alpha_1$.  Every equiangular hyperbolic $4$-simplex with angle
$\alpha$ is isometric to $\Delta^4_\alpha$.
\end{proposition}

\begin{proof}
The matrix in \eqref{eq:simplex-Gram} has signature $(4,1)$ precisely when
$1-4c<0$, or $\alpha<\alpha_1$.  Its four-by-four vertex submatrices are
positive definite for $c<1/3$ and parabolic of rank $3$ for $c=1/3$.
Finally, the common dihedral angle determines every entry of the Gram matrix,
and the Gram matrix determines the simplex up to hyperbolic isometry.
\end{proof}

\subsection{The regular equiangular 4-cube}

The existence of the second model is particularly transparent in the Klein
ball.  For $0<r\leq1/2$, consider the Euclidean cube
\[
 [-r,r]^4\subset\overline{\mathbb B^4}.
\]
The outward Lorentzian unit normals to the two facets $x_i=\pm r$ can be
chosen as
\[
 u_i^\pm=\frac{(r,\pm e_i)}{\sqrt{1-r^2}}.
\]
For $i\neq j$ one has
\[
 \langle u_i^\varepsilon,u_j^\delta\rangle
 =-\frac{r^2}{1-r^2}.
\]
Consequently all adjacent facets meet at angle $\alpha$ if
\begin{equation}\label{eq:cube-radius}
 r^2=\frac{c}{1+c},
 \qquad c=\cos\alpha.
\end{equation}
All vertices lie in the closed Klein ball exactly when $4r^2\leq1$, which is
equivalent to $c\leq1/3$.

\begin{proposition}\label{prop:cube-existence-uniqueness}
For every
\[
 \alpha_0\leq\alpha<\frac{\pi}{2}
\]
there exists a finite-volume regular equiangular hyperbolic $4$-cube
$C^4_\alpha$.  It is ideal at $\alpha=\alpha_0$ and compact for
$\alpha>\alpha_0$.  Moreover, if $\alpha>\alpha_0$, every equiangular
hyperbolic $4$-polytope with angle $\alpha$ that is combinatorially
equivalent to the $4$-cube is isometric to $C^4_\alpha$.
\end{proposition}

\begin{proof}
Existence and the assertion about vertices follow from
\eqref{eq:cube-radius}.  If $\alpha>\alpha_0$, the polytope is compact by
Proposition~\ref{prop:simple-angle-range}.  Andreev's rigidity theorem states
that a compact non-obtuse hyperbolic polytope of dimension at least three is
determined, up to isometry, by its face lattice and its dihedral angles; see
\cite{Andreev1970Convex}.  It therefore gives the
asserted uniqueness in the compact range.  Notice that no uniqueness claim
at the ideal endpoint $\alpha=\alpha_0$ is needed below.
\end{proof}

\section{The flag dichotomy}\label{sec:flag}

Let $P^*$ denote the combinatorial dual of the simple polytope $P$, and put
\[
 \Delta=\partial P^*.
\]
Then $\Delta$ is a simplicial $3$-sphere.  A set of vertices of a simplicial
complex is a \emph{missing face} if it does not span a simplex but every
proper subset does.  The complex is \emph{flag} if all its missing faces have
two vertices.

\begin{lemma}[Flag dichotomy]\label{lem:flag-dichotomy}
Let $P\subset\HH^4$ be as in Proposition~\ref{prop:simple-angle-range}.  Then
either $P$ is a $4$-simplex or $\Delta=\partial P^*$ is flag.
\end{lemma}

\begin{proof}
Suppose that $\Delta$ has a missing face with $k\geq3$ vertices.  The
corresponding $k$ facets of $P$ meet pairwise, so their Gram matrix is
\begin{equation}\label{eq:k-face-Gram}
 G_k=(1+c)I_k-cJ_k.
\end{equation}
Its eigenvalues are $1+c$ with multiplicity $k-1$ and
$1-(k-1)c$ with multiplicity one.

If $k\leq4$ and $\alpha>\alpha_0$, then $G_k$ is positive definite.  The
elliptic face criterion recalled in Section~\ref{sec:geometry} implies that the corresponding
facets have a common face, contrary to the missing-face assumption.  At
$\alpha=\alpha_0$, the same argument applies for $k\leq3$.  For $k=4$, the
matrix is indecomposable and parabolic of rank $3$, with positive kernel
vector $(1,1,1,1)$.  The parabolic face
criterion gives a common ideal vertex, again a contradiction.

Since $\Delta$ has dimension $3$, a missing face has at most five vertices:
if $k\geq6$, any five of its vertices form a proper subset and would have to
span a $4$-simplex in the $3$-dimensional complex $\Delta$.
It remains to consider $k=5$.  In this case the boundary of a $4$-simplex is
a subcomplex of $\Delta$.  It is an embedded simplicial $3$-sphere.  By
invariance of domain, its image is open in the connected $3$-manifold
$|\Delta|$; it is also closed.  Hence it is all of $\Delta$.  Thus $P$ is a
$4$-simplex.
\end{proof}

We now recall the dimension-three case of the Charney--Davis conjecture.

\begin{theorem}[Davis--Okun \cite{DavisOkun2001}]\label{thm:davis-okun}
If $\Delta$ is a flag simplicial $3$-sphere, then
\begin{equation}\label{eq:charney-davis}
 1-\frac{f_0(\Delta)}2+\frac{f_1(\Delta)}4
 -\frac{f_2(\Delta)}8+\frac{f_3(\Delta)}{16}\geq0.
\end{equation}
\end{theorem}

Using the Dehn--Sommerville relations for a simplicial $3$-sphere,
\eqref{eq:charney-davis} is equivalent to
\begin{equation}\label{eq:CD-short}
 f_3(\Delta)-4f_0(\Delta)+16\geq0.
\end{equation}
Duality therefore gives the following form that will enter the volume
comparison.

\begin{corollary}\label{cor:f-inequality}
If $P$ is not a simplex, then
\begin{equation}\label{eq:primal-CD}
 f_0(P)-4f_3(P)+16\geq0,
\end{equation}
or equivalently
\begin{equation}\label{eq:primal-CD-shift}
 f_0(P)-16\geq4\bigl(f_3(P)-8\bigr).
\end{equation}
\end{corollary}

We shall also need the equality case for the smallest possible number of
vertices of a flag $3$-sphere.

\begin{lemma}\label{lem:flag-eight}
Every flag simplicial $3$-sphere has at least eight vertices.  If it has
exactly eight vertices, it is the boundary of the $4$-dimensional cross
polytope.
\end{lemma}

\begin{proof}
The link of a vertex $v$ is a flag triangulation of $\SSph^2$.  Such a
triangulation has at least six vertices; with six vertices it is the boundary
of the octahedron.  Indeed, flagness rules out vertices of degree three.
For a triangulated $2$-sphere, Euler's relation gives
$2f_1=6f_0-12$; hence minimum degree at least four implies $f_0\geq6$.
When $f_0=6$, every vertex has degree four, so the missing edges form a
perfect matching in $K_6$.  The $1$-skeleton is therefore the octahedral
graph, and flagness identifies the complex with the boundary of the
octahedron.

Since $\Delta$ is flag, if every vertex belonged to the closed star of $v$, then for every simplex $\sigma$ not containing $v$, the vertices of $\sigma \cup \{v\}$ would form a clique and hence a simplex. 
Thus $\Delta$ would be a cone with apex $v$, which is impossible for a simplicial sphere. Therefore some vertex lies outside $\st(v)$.
Thus there are at least
$1+6+1=8$ vertices.  Here the \emph{antistar} of $v$ is the subcomplex
consisting of the simplices not containing $v$.  If equality holds, the
antistar has a single vertex $w$ outside its boundary $\lk(v)$.  No
tetrahedron in the antistar can have all four vertices in $\lk(v)$: those
four vertices, together with $v$, would form a clique of five vertices,
contradicting flagness and the dimension of $\Delta$.  Therefore every
tetrahedron of the antistar contains $w$, so the antistar is the cone
$w*\lk(v)$.  Hence the whole complex is the suspension of the octahedral
$2$-sphere, which is the boundary of the $4$-cross-polytope.
\end{proof}

In primal language, Lemma~\ref{lem:flag-eight} says that every nonsimplex in
our class has at least eight facets; equality forces the combinatorial type
of the $4$-cube.

\section{The Gram--Euler volume formula}\label{sec:volume}

\subsection{The regular spherical link}

For $\alpha_0\leq\alpha<\pi/2$, let $L_\alpha$ be the regular spherical
tetrahedron with dihedral angle $\alpha$; at $\alpha=\alpha_0$ it is
understood as the degenerate Euclidean limit.  Define its normalized volume
by
\begin{equation}\label{eq:omega-def}
 \omega(\alpha)=\frac{\Vol_{\SSph^3}(L_\alpha)}{\Vol(\SSph^3)}
 =\frac{\Vol_{\SSph^3}(L_\alpha)}{2\pi^2}.
\end{equation}
Let $\beta(\alpha)$ be the common spherical edge length of $L_\alpha$.
Inverting the matrix in \eqref{eq:link-Gram}, or applying spherical duality,
gives
\begin{equation}\label{eq:beta}
 \cos\beta(\alpha)=\frac{\cos\alpha}{1-2\cos\alpha}.
\end{equation}
Thus $\beta$ increases from $0$ to $\pi/2$ as $\alpha$ increases from
$\alpha_0$ to $\pi/2$, and
\begin{equation}\label{eq:beta-alpha1}
 \beta(\alpha_1)=\frac{\pi}{3}.
\end{equation}

The spherical Schl\"afli formula for a $3$-simplex is
\[
 d\Vol_{\SSph^3}=\frac12\sum_e \ell_e\,d\theta_e;
\]
see, for example, \cite{Milnor1994}.  Since $L_\alpha$ has six equal edges,
we obtain
\begin{equation}\label{eq:omega-derivative}
 \omega'(\alpha)=\frac{3\beta(\alpha)}{2\pi^2}.
\end{equation}
At $\alpha=\pi/2$, the link is a spherical orthant, so
\begin{equation}\label{eq:omega-endpoint}
 \omega\left(\frac{\pi}{2}\right)=\frac1{16}.
\end{equation}

\subsection{Volume as a function of the face vector}

The non-Euclidean Gram--Euler relation goes back to Sommerville
\cite{Sommerville1927}; see also \cite{Heckman1995} for the same
normalization in the Coxeter case.  For a proper face $F$ of $P$, let
$\angle(F)$ denote the volume of the transverse tangent cone of $P$ along $F$, divided by the
volume of the corresponding unit sphere.  At an ideal vertex this normalized
angle is understood to be zero.  In dimension four the Gram--Euler relation
for a compact convex hyperbolic polytope gives
\begin{equation}\label{eq:Gram-Euler-general}
 \frac{3}{4\pi^2}\Vol(P)
 =1-\sum_{F^3}\angle(F^3)+\sum_{F^2}\angle(F^2)
 -\sum_{F^1}\angle(F^1)+\sum_{F^0}\angle(F^0).
\end{equation}
For a simple equiangular $4$-polytope the four kinds of proper face angles
are, respectively,
\[
 \frac12,
 \qquad
 \frac{\alpha}{2\pi},
 \qquad
 \frac{3\alpha-\pi}{4\pi},
 \qquad
 \omega(\alpha).
\]
Indeed, the normal link at an edge is a spherical triangle with all three
angles equal to $\alpha$, and hence has area $3\alpha-\pi$.  Formula
\eqref{eq:Gram-Euler-general} therefore becomes
\begin{equation}\label{eq:Gram-Euler-f}
 \frac{3}{4\pi^2}\Vol(P)
 =1-\frac{f_3}2+\frac{\alpha}{2\pi}f_2
 -\frac{3\alpha-\pi}{4\pi}f_1+\omega(\alpha)f_0.
\end{equation}
The same formula holds at the finite-volume endpoint
$\alpha=\alpha_0$.  For completeness, truncate every cusp by a totally
geodesic hyperplane and let the truncation hyperplanes tend to the
corresponding ideal points.  The compact Gram--Euler relation applies to the
truncated polytopes, whose volumes converge to $\Vol(P)$.  The terms carried
by the old faces converge to those in
\eqref{eq:Gram-Euler-general}.  At each cusp, the alternating sum of the
terms carried by the new truncation face and its subfaces converges to the
Euclidean Gram--Euler sum of the horospherical link, and hence is zero.
Thus passage to the limit proves \eqref{eq:Gram-Euler-f}; in particular, the
ideal vertex-angle term is zero.

Using \eqref{eq:simple-f-relations}, set
\begin{equation}\label{eq:a-b-def}
 a(\alpha)=\omega(\alpha)+\frac12-\frac{\alpha}{\pi},
 \qquad
 b(\alpha)=\frac12-\frac{\alpha}{2\pi}.
\end{equation}
We obtain the key formula
\begin{equation}\label{eq:volume-linear}
 v(P):=\frac{3}{4\pi^2}\Vol(P)
 =1+a(\alpha)f_0(P)-b(\alpha)f_3(P).
\end{equation}

For the two model polytopes this gives
\begin{equation}\label{eq:volume-simplex}
 v(\Delta^4_\alpha)=1+5a(\alpha)-5b(\alpha),
\end{equation}
and
\begin{equation}\label{eq:volume-cube}
 v(C^4_\alpha)=1+16a(\alpha)-8b(\alpha).
\end{equation}

\subsection{The coefficient inequalities}

The following elementary analytic facts contain all the angle dependence of
the proof.

\begin{lemma}\label{lem:coefficients}
For $\alpha_0\leq\alpha<\pi/2$, one has
\begin{equation}\label{eq:a-positive}
 a(\alpha)>0
\end{equation}
and
\begin{equation}\label{eq:4a-b-positive}
 4a(\alpha)-b(\alpha)>0.
\end{equation}
Moreover, for $\alpha_0\leq\alpha\leq\alpha_1$,
\begin{equation}\label{eq:11a-3b-positive}
 11a(\alpha)-3b(\alpha)>0.
\end{equation}
\end{lemma}

\begin{proof}
By \eqref{eq:omega-derivative},
\[
 a'(\alpha)=\frac{3\beta(\alpha)}{2\pi^2}-\frac1\pi<0
\]
for $\alpha<\pi/2$.  Equations \eqref{eq:omega-endpoint} and
\eqref{eq:a-b-def} give $a(\pi/2)=1/16$, proving
\eqref{eq:a-positive}.

Put $h=4a-b$.  Then
\[
 h'(\alpha)=\frac{6\beta(\alpha)}{\pi^2}-\frac{7}{2\pi}<0,
\]
because $\beta(\alpha)<\pi/2$ gives
$h'(\alpha)<-1/(2\pi)$.  On the other hand,
$h(\pi/2)=4/16-1/4=0$.  Hence $h(\alpha)>0$ for
$\alpha<\pi/2$.

Finally, put $d=11a-3b$.  In the simplex range,
\eqref{eq:beta-alpha1} gives $\beta(\alpha)\leq\pi/3$, and therefore
\[
 d'(\alpha)=\frac{33\beta(\alpha)}{2\pi^2}
 -\frac{19}{2\pi}<0.
\]
Indeed, $\beta(\alpha)\leq\pi/3$ gives
$d'(\alpha)\leq-4/\pi$.
At $\alpha=\alpha_1$, the simplex is Euclidean-degenerate, so the continuous
extension of \eqref{eq:volume-simplex} is zero.  Hence
\[
 d(\alpha_1)=v(C^4_{\alpha_1})-v(\Delta^4_{\alpha_1})
 =v(C^4_{\alpha_1})>0.
\]
Since $d$ is decreasing, this proves \eqref{eq:11a-3b-positive} on the whole
interval $[\alpha_0,\alpha_1]$.
\end{proof}

\section{Proof of the main theorem}\label{sec:proof}

We first compare every nonsimplex with the cube.

\begin{proposition}\label{prop:nonsimplex-cube}
Let $P\subset\HH^4$ be a finite-volume equiangular polytope with
$\alpha_0\leq\alpha<\pi/2$.  If $P$ is not a simplex, then
\[
 \Vol(P)\geq\Vol(C^4_\alpha).
\]
Equality forces $P$ to be combinatorially equivalent to the $4$-cube.  If
$\alpha>\alpha_0$, equality holds only if $P$ is isometric to
$C^4_\alpha$.
\end{proposition}

\begin{proof}
Subtracting \eqref{eq:volume-cube} from \eqref{eq:volume-linear} gives
\begin{align*}
 v(P)-v(C^4_\alpha)
 &=a(\alpha)\bigl(f_0(P)-16\bigr)
   -b(\alpha)\bigl(f_3(P)-8\bigr)\\
 &\geq\bigl(4a(\alpha)-b(\alpha)\bigr)
      \bigl(f_3(P)-8\bigr).
\end{align*}
Here the inequality is Corollary~\ref{cor:f-inequality}, together with
$a(\alpha)>0$.  By Lemma~\ref{lem:flag-eight}, $f_3(P)\geq8$, and by
Lemma~\ref{lem:coefficients}, $4a-b>0$.  This proves the volume inequality.

If $f_3(P)>8$, the inequality is strict.  If $f_3(P)=8$, the dual flag
$3$-sphere has eight vertices and is the boundary of the $4$-cross-polytope
by Lemma~\ref{lem:flag-eight}.  Thus $P$ is combinatorially a $4$-cube.
Proposition~\ref{prop:cube-existence-uniqueness} then shows that $P$ is
isometric to $C^4_\alpha$ whenever $\alpha>\alpha_0$.
\end{proof}

We can now complete the argument.

\begin{proof}[Proof of Theorem~\ref{thm:main}]
First suppose that $\alpha_0\leq\alpha<\alpha_1$.  If $P$ is a simplex,
Proposition~\ref{prop:simplex-existence} shows that it is isometric to
$\Delta^4_\alpha$.  If $P$ is not a simplex, Proposition
\ref{prop:nonsimplex-cube} and equations
\eqref{eq:volume-simplex}--\eqref{eq:volume-cube} give
\[
 v(P)\geq v(C^4_\alpha)
 =v(\Delta^4_\alpha)+11a(\alpha)-3b(\alpha)
 >v(\Delta^4_\alpha),
\]
where the last inequality is Lemma~\ref{lem:coefficients}.  This proves the
first part, including uniqueness.

Now suppose that $\alpha_1\leq\alpha<\pi/2$.  An equiangular hyperbolic
$4$-simplex does not exist in this range by
Proposition~\ref{prop:simplex-existence}.  Hence every $P$ is a nonsimplex,
and Proposition~\ref{prop:nonsimplex-cube} applies directly.  Its equality
statement proves uniqueness of the minimum.
\end{proof}

\begin{remark}\label{rem:stronger-form}
Proposition~\ref{prop:nonsimplex-cube} is slightly stronger than the main
theorem: throughout the entire angle range
$[\alpha_0,\pi/2)$, the regular $4$-cube is a minimum among equiangular
$4$-polytopes that are not simplices, and it is the unique minimum up to
isometry for $\alpha>\alpha_0$.  The global minimum has a genuine jump at
$\alpha_1$: as $\alpha\uparrow\alpha_1$, the volume of the regular
$4$-simplex tends to zero, but at $\alpha=\alpha_1$ that simplex no longer
exists and the minimum is the positive volume of $C^4_{\alpha_1}$.
\end{remark}

\section{A higher-dimensional question}\label{sec:higher}

The two critical angles and the two model families have direct analogues in
every dimension.  For $n\geq4$, put
\[
 \alpha_0(n)=\arccos\frac1{n-1},
 \qquad
 \alpha_1(n)=\arccos\frac1n.
\]
The Gram matrix of the regular $n$-simplex has eigenvalues $1+c$ with
multiplicity $n$ and $1-nc$ with multiplicity one.  Thus the finite-volume
regular hyperbolic $n$-simplex exists precisely for
\[
 \alpha_0(n)\leq\alpha<\alpha_1(n).
\]
Similarly, the Klein cube $[-r,r]^n$ with
$r^2=c/(1+c)$ has finite volume precisely when
$nr^2\leq1$, or $\alpha\geq\alpha_0(n)$.

This suggests the following statement.

\begin{conjecture}\label{conj:higher}
Let $n\geq4$ be even and let $P\subset\HH^n$ be a finite-volume convex
equiangular polytope with common dihedral angle $\alpha<\pi/2$.
\begin{enumerate}
\item For $\alpha_0(n)\leq\alpha<\alpha_1(n)$, the regular hyperbolic
$n$-simplex is the unique minimum-volume polytope.
\item For $\alpha_1(n)\leq\alpha<\pi/2$, the regular equiangular hyperbolic
$n$-cube is the unique minimum-volume polytope.
\end{enumerate}
\end{conjecture}

The proof above does not formally extend to $n\geq6$.  Although the
Gram--Euler relation still contains volume in even dimension, it involves
several independent sums of intermediate face angles.  Moreover, the
four-dimensional proof uses exactly the Charney--Davis inequality for flag
$3$-spheres, the first nontrivial case in which the conjectured inequality is
a theorem for all flag spheres.  In higher odd dimensions the corresponding
Charney--Davis statement is open in general.  Thus Theorem~\ref{thm:main}
may be viewed as the first case of Conjecture~\ref{conj:higher} in which the
required combinatorial sign theorem is presently available.

\section*{Statements and Declarations}

\noindent\textbf{Funding.}
The author was supported by the state task of the Sobolev Institute of
Mathematics, project No.~FWNF-2026-0011.

\noindent\textbf{Use of AI tools.}
The author used OpenAI's ChatGPT as a research and verification aid. The system contributed to the development and checking of the proof strategy and to the preparation of the exposition. The author takes full responsibility for the final manuscript.

\end{document}